\documentclass[a4paper,11pt]{article}
\usepackage{amsthm,amsmath,amssymb,amsfonts,mathrsfs,hyperref,microtype}
\usepackage{graphicx,color}

\newtheorem{defn}{Definition}[section]
\newtheorem{thm}{Theorem}[section]
\newtheorem{prop}{Proposition}[section]

\newtheorem{rmk}{Remark}[section]
\newtheorem{lma}{Lemma}[section]

\newtheorem{exm}{Example}[section]
\def\N{{\rm I\kern-0.16em N}}
\def\R{{\rm I\kern-0.16em R}}
\def\E{{\rm I\kern-0.16em E}}
\def\P{{\rm I\kern-0.16em P}}
\def\F{{\rm I\kern-0.16em F}}
\def\B{{\rm I\kern-0.16em B}}
\def\C{{\rm I\kern-0.46em C}}
\def\G{{\rm I\kern-0.50em G}}
\newcommand{\ud}{\mathrm{d}}

\newcommand{\law}{\stackrel{\text{law}}{=}}
\newcommand{\hnorm}[1]{\Vert #1\Vert_{\tilde{\mathcal{H}}}}
\numberwithin{equation}{section}
\font\eka=cmex10
\usepackage{ae}
\def\ind{\mathrel{\hbox{\rlap{%
\hbox to 7.5pt{\hrulefill}}\raise6.6pt\hbox{\eka\char'167}}}}
\begin{document}
\title{\textbf{Parameter estimation based on discrete observations of fractional Ornstein-Uhlenbeck process of the second kind}}
\author{Ehsan Azmoodeh \and Lauri Viitasaari}
\renewcommand{\thefootnote}{\fnsymbol{footnote}}

\author{Ehsan Azmoodeh\footnotemark[1] \, and \, Lauri Viitasaari\footnotemark[2]}

\footnotetext[2]{Department of Mathematics and System Analysis, Aalto University School of Science, Helsinki P.O. Box 11100, FIN-00076 Aalto,  FINLAND.}

\footnotetext[1]{Mathematics Research Unit, Luxembourg University, P.O. Box L-1359, Luxembourg, ehsan.azmoodeh@uni.lu.}
\maketitle

\begin{abstract}
Fractional Ornstein-Uhlenbeck process of the second kind $(\text{fOU}_{2})$ is solution of the Langevin equation $\mathrm{d}X_t = -\theta X_t\,\mathrm{d}t+\mathrm{d}Y_t^{(1)}, \ \theta >0$ with Gaussian driving 
noise $ Y_t^{(1)} := \int^t_0 e^{-s} \,\mathrm{d}B_{a_s}$, where $ a_t= H e^{\frac{t}{H}}$ and $B$ is a fractional Brownian motion 
with Hurst parameter $H \in (0,1)$. In this article, we consider the case $H>\frac{1}{2}$. Then using the ergodicity of $\text{fOU}_{2}$ process, we construct consistent estimators of drift parameter $\theta$ based 
on discrete observations in two possible cases: $(i)$ the Hurst parameter $H$ is known and $(ii)$ the Hurst parameter $H$ is unknown. Moreover, using Malliavin calculus technique, we prove central limit theorems for our estimators which is valid for the 
whole range $H \in (\frac{1}{2},1)$.

\medskip


\smallskip

\end{abstract}

\tableofcontents

\section{Introduction}

\subsection{Motivation and overwive}
Assume $B=\{B_t\}_{ t\geq 0}$ is a fractional Brownian motion with Hurst parameter $H \in (0,1)$, i.e a continuous, 
centered Gaussian process with covariance function
\begin{equation*}
R_{H}(s,t)= \frac{1}{2}\{ s^{2H} + t^{2H} - \vert t -s \vert^{2H} \}, \quad s,t \ge 0.
\end{equation*}

Consider the following Langevin equation with drift parameter $\theta >0$ and driving noise $N$
\begin{equation}\label{LfOU2}
 \mathrm{d}X_t = -\theta X_t\,\mathrm{d}t+\mathrm{d} N_t.
 \end{equation}

When the driving noise $N=B$ is fractional Brownian motion, a solution of the Langevin equation $(\ref{LfOU2})$ is called the fractional Ornstein-Uhlenbeck process of the first kind, 
in short $(\text{fOU}_{1})$. The fractional Ornstein-Uhlenbeck process of the second kind is a solution of the 
Langevin equation $(\ref{LfOU2})$ when the driving noise $N_t = Y^{(1)}_t := \int^t_0 e^{-s} \,\mathrm{d}B_{a_s}$ and $\ a_t= H e^{\frac{t}{H}}$. The terms \textit{``of the first kind''} and 
\textit{``of the second kind''} are taken from Kaarakka \& Salminen \cite{k-s}. It is well known that the classical Ornstein-Uhlenbeck process, i.e. when the driving noise $N=W$ is a standard Brownian motion, 
has the same finite dimensional distributions as the Lamperti transformation (see \ref{eq:lamperti} for definition) of Brownian motion. Surprisingly, when one 
replaces Brownian motion with fractional Brownian motion the solution of the Langevin equation $(\ref{LfOU2})$ is completely different from the one that is obtained by the Lamperti 
transformation of fractional Brownian motion, see \cite{C-K-M, k-s}. The motivation behind introducing the noise process 
$N=Y^{(1)}$ is related to the Lamperti transformation of fractional Brownian motion. We refer to Subsection \ref{fou2} and in more details to \cite[Section $3$]{k-s}.\\

Typically, statistical models with fractional processes exhibit short (or long) memory property whether $H<\frac{1}{2}$ (or $H>\frac{1}{2}$). However, from statistical point of view, regardless of the range of 
Hurst parameter $H$, the $\text{fOU}_{2}$ process unlike the $\text{fOU}_{1}$ process always exhibits short memory property. This phenomenon makes $\text{fOU}_{2}$ an interesting process for modeling 
applications in many different disciplines. For example, for applications of short memory processes in econometric or in modeling the extremes of time series see \cite{short-econometric,ch-da} respectively. \\

In this article, we take advantage of the ergodicity of $\text{fOU}_{2}$ process to construct consistent estimator of the drift parameter $\theta$ based on observations of the process at discrete times. Assume 
that we observe the process at discrete times $0, \Delta_N, 2\Delta_N, \cdots, N\Delta_N$ and let $T_N = N\Delta_N$ denote the length of the observation window. Our aim is to show that: \\

 \textbf{(i) when $H$ is known} one can construct a strongly consistent estimator $\widehat{\theta}$, introduced in Theorem \ref{main_theorem}, with asymptotic normality property under the mesh conditions
 \begin{equation*}
 T_N \to \infty, \quad \text{and} \quad N\Delta_{N}^{2} \to 0
 \end{equation*}
with arbitrary mesh $\Delta_N$ such that $\Delta_N \to 0$ as $N$ tends to infinity.\\

\textbf{(ii) when $H$ is unknown} one can construct another strongly consistent estimator $\widetilde{\theta}$, introduced in Theorem \ref{thm:main2}, with asymptotic normality property under the restricted mesh condition
\begin{equation*}
\Delta_N = N^{-\alpha}, \quad \text{with} \quad \alpha \in (\frac{1}{2}, \frac{1}{4H-2}\wedge 1).
\end{equation*}

 \subsection{History and further motivations}
 Statistical inferences of drift parameter $\theta$ based on data recorded from continuous $($discrete$)$ trajectories of $X$ is an interesting problem in the realm of mathematical statistics. In the case of 
 diffusion processes with Brownian motion as driving noise the problem is well studied. See for example \cite{kut} and references therein among many others. The problem of estimation of drift parameter 
 becomes very challenging with fractional processes as driving noise. This is mainly because of the fact that fractional Brownian motion $B$ with Hurst parameter $H \neq \frac{1}{2}$ is neither a semimartingale 
 nor a Markov process. We refer to the recent book \cite{rao} for more details in this regards. In the case of fractional Ornstein-Uhlenbeck process of the first kind, the two popular statistical estimators, 
 namely maximum likelihood $(\text{MLE})$ and least squares $(\text{LSE})$  estimators based on continuous observations of the process are considered in Kleptsyna \& Breton \cite{k-l} and Hu \& Nualart \cite{h-n} 
 respectively. In this case it turns out that \text{MLE} and LSE provide strongly consistent estimators. Moreover, the asymptotic normality of MLE is shown in \cite{b-c-s} when $H > \frac{1}{2}$  and for LSE 
 in \cite{h-n} when $H \in [\frac{1}{2}, \frac{3}{4})$. In the case of fractional Ornstein-Uhlenbeck process of the second kind, Azmoodeh \& Morlanes \cite{a-m} showed that LSE is a consistent estimator using 
 continuous observations. Moreover, they showed that a central limit theorem for LSE holds for the whole range $H > \frac{1}{2}$.\\
 
The main feature of this paper is to provide strongly consistent estimators of drift parameter $\theta$ based on discrete observations of the process $X$, and more importantly to show they satisfy CLTs using the modern 
approach of Malliavin calculus for normal approximations \cite{n-p}. It is very important from practical point of view to assume that we have a data collected from process $X$ observed at discrete times. 
In addition to its applicability, such a demand makes the problem more delicate. Therefore, such problem could not be remained open for the fractional 
Ornstein-Uhlenbeck process of the first kind. In fact for the $\text{fOU}_{1}$ process, estimation of drift parameter $\theta$ with discretization procedure of integral transform is considered in 
Xiao et. al. \cite{x-z-x} assuming that Hurst parameter $H$ is known. In the same setup, Brouste \& Iacus \cite{b-i} introduce an estimation procedure that can estimate both drift parameter 
$\theta$ and Hurst parameter $H$ based on discrete observations. In this paper, we also display a new estimation method that can estimate drift parameter 
$\theta$ of the $\text{fOU}_{2}$ process based on discrete observations when Hurst parameter $H$ is unknown (Theorem \ref{thm:main2}).\\

\subsection{Plan}
The paper is organized as follows. In section 2, we give auxiliary facts on Malliavin calculus and fractional Ornstein-Uhlenbeck processes. Section \ref{Hknown} is devoted to estimation of drift parameter when $H$ is known.
In section \ref{estimationH}, we give a short explanation of estimation of Hurst parameter $H$ based on discrete observations. Section \ref{unknownH} deals with estimation of drift parameter when $H$ is unknown. We also collect
all technical computations to appendix A.\\

\section{Auxiliary facts}

\subsection{A brief review on Malliavin calculus}\label{Malliavin}
In this subsection, we briefly introduce some basic facts on Malliavin calculus with respect to Gaussian processes. Also the use of Malliavin calculus to obtain central limit theorem for a sequence of multiple Wiener 
integrals is now well established.  For more details, we refer to \cite{a-m-n,nu1,n-p}. Let $W$ be a Brownian motion. Assume that $G=\{ G_t \}_{t\in [0,T]}$ a continuous centered Gaussian process of the form 
\begin{equation*}
 G_t = \int_{0}^{t} K(t,s) \ud W_s
\end{equation*}
where the \textit{Volterra} kernel $K$, meaning that $K(t,s)=0$ for all $s>t$, satisfies $\sup_{t \in [0,T]} \int_{0}^{t}K(t,s)^{2} \ud s < \infty$. Moreover, we assume that for any $s$ the function $K(\cdot,s)$ 
is bounded variation on any interval $(u,T]$ for all $u >s$. A typical example of this type of Gaussian processes is fractional Brownian motion $B$. It is known that when $H>\frac{1}{2}$, the kernel takes the form 
\begin{equation*}
K_{H}(t,s)= c_H s^{\frac{1}{2} - H} \int_{s}^{t} (u - s)^{H - \frac{3}{2}} u^{H - \frac{1}{2}} \ud u.
\end{equation*}
Moreover, we have the following inverse relation 
\begin{equation}\label{BW}
 W_t = B \big( (K^{*}_{H})^{-1} (\textbf{1}_{[0,t]})\big)
\end{equation}
where the operator $K^{*}_{H}$ is defined as 

\begin{equation*}
 (K^{*}_{H}\varphi)(s)= \int_{s}^{T} \varphi(t) \frac{\partial K_H}{\partial t} (t,s) \ud t.
\end{equation*}

Consider the set $\mathcal{E}$ of all step functions on $[0,T]$. The Hilbert space $\mathcal{H}$ associated to process $G$ is the closure of $\mathcal{E}$ with respect
to inner product
\begin{equation*}
 \langle \mathbf{1}_{[0,t]},\mathbf{1}_{[0,s]} \rangle_{\mathcal{H}} = R_{G}(t,s)
\end{equation*}
where $R_G(t,s)$ denotes the covariance function of $G$. Then the mapping $\mathbf{1}_{[0,t]} \mapsto G_t$ can be extended to an isometry between Hilbert space $\mathcal{H}$ and Gaussian space $\mathcal{H}_1$ 
associated with Gaussian process 
$G$. Consider the space $\mathcal{S}$ of all smooth random variables of the form 
\begin{equation}\label{eq:smooth}
F= f(G(\varphi_1), \cdots, G(\varphi_n)), \qquad \varphi_1, \cdots, \varphi_n \in \mathcal{H},
\end{equation}
where $f \in C_{b}^{\infty}(\R^n)$. For any smooth random variable $F$ of the form (\ref{eq:smooth}), we define its Malliavin derivative $D^{(G)}= D $ as an element of 
$L^{2}(\Omega;\mathcal{H})$ by

\begin{equation*}
D F= \sum_{i=1}^{n} \partial_{i} f (G(\varphi_1), \cdots, G(\varphi_n)) \varphi_i.
\end{equation*}
In particular, $D G_t = \mathbf{1}_{[0,t]}$. We denote by $\mathbb{D}^{1,2}_{G}= \mathbb{D}^{1,2}$ the Hilbert space of all square integrable Malliavin derivative random variables as the 
closure of the set $\mathcal{S}$ of smooth random variables with respect to the norm 

\begin{equation*}
\Vert F \Vert_{1,2}^{2} = \E |F|^{2} + \E ( \Vert D F \Vert_{\mathcal{H}} ^{2}).
\end{equation*}


Consider the linear operator $K^{*}$ from $\mathcal{E}$ to $L^{2}[0,T]$ defined by
\begin{equation*}
 (K^{*} \varphi)(s) = \varphi(s)K(T,s) + \int_{s}^{T} \left[ \varphi(t) - \varphi(s) \right] K(\ud t,s).
\end{equation*}
Here, $K(\ud t,s)$ stands for the measure associated to the bounded variation function $K(\cdot,s)$. The Hilbert space $\mathcal{H}$ generated by covariance function of the Gaussian process $G$ can be represented as $\mathcal{H} = (K^{*})^{-1} (L^{2}[0,T])$ and 
$\mathbb{D}^{1,2}_{G}(\mathcal{H}) =(K^{*})^{-1} \big(\mathbb{D}^{1,2}_{W}(L^{2}[0,T])\big)$. For any $n \ge 1$, let $\mathscr{H}_n$ be the $n$th Wiener chaos of $G$, i.e.
the closed linear subspace of $L^2 (\Omega)$ generated by the random variables 
$\{ H_n \left( G(\varphi) \right),\ \varphi \in \mathcal{H}, \ \Vert \varphi \Vert_{\mathcal{H}} = 1\}$ where $H_n$ is the $n$th Hermite polynomial. It is well known that
the mapping $I_{n}^{G}(\varphi^{\otimes n}) = n! H_n \left( G(\varphi)\right)$ provides a linear isometry between the symmetric tensor product $\mathcal{H}^{\odot n}$
and subspace $\mathscr{H}_n$. The random variables $I_{n}^{G}(\varphi^{\otimes n})$ are customary called \textit{multiple Wiener} integrals of order $n$ with respect to Gaussian process $G$. When $G$ is Brownian 
motion, the random variables $I_{n}^{G}$ coincide with multiple It$\hat{o}$ integrals.\\ 

The next proposition provides a central limit theorem for a sequence of multiple Wiener integrals of fixed order. Let $\mathcal{N}(0,\sigma^2)$ denote the Gaussian distribution with zero mean and variance 
$\sigma^2$. The notation $\overset{\text{law}}\longrightarrow$ stands for convergence in distribution.


\begin{prop}\cite{n-o}\label{CLT} Let $\{F_n\}_{n \ge 1}$ be a sequence of random variables in the $q$th Wiener chaos $\mathscr{H}_q$  with $q \ge2$ such that $\lim_{n \to \infty} \E(F_n ^2) = \sigma^2$. Then the following 
statements are equivalent:
 \begin{itemize}
 \item[(i)] $F_n \overset{\text{law}}\longrightarrow \mathcal{N}(0,\sigma^2)$ as $n$ tends to infinity.
 
 \item[(ii)] $ \Vert DF_n \Vert^{2}_{\mathcal{H}}$ converges in $L^{2}(\Omega)$ to $q \sigma^{2}$ as $n$ tends to infinity.
 \end{itemize}
\end{prop}



\subsection{Fractional Ornstein-Uhlenbeck processes}\label{fOU}

In this subsection, we briefly introduce fractional Ornstein-Uhlenbeck processes. The main references are \cite{C-K-M,k-s}. We mostly focus on fractional Ornstein-Uhlenbeck process of the second kind. 
Moreover, we provide some new results on fractional Ornstein-Uhlenbeck process of the second kind. 

\subsubsection{Fractional Ornstein-Uhlenbeck processes of the first kind}\label{fou1}
Let $B=\{B_t\}_{ t\geq 0}$ be a fractional Brownian motion with Hurst parameter $H \in (0,1)$. To obtain fractional Ornstein-Uhlenbeck process, consider the following Langevin equation

\begin{equation}\label{Langevin-first}
 dU^{(H,\xi_0)}_t =-\theta U^{(H,\xi_0)}_t dt + dB_t, \quad U^{(H,\xi_0)}_0=\xi_0.
\end{equation}

 The solution of the SDE $(\ref{Langevin-first})$ can be expressed as

\begin{equation}\label{generalfOU1}
U^{(H,\xi_0)}_t = e^{-\theta t} \left( \xi_{0} + \int^t_0 e^{\theta s}\,\ud B_s\right).
\end{equation}

Notice that the stochastic integral can be understood as a pathwise Riemann-Stieltjes integral or, equivalently, as Wiener integral. Let $\hat{B}$ denote a two sided fractional Brownian motion. 
The special selection
\begin{equation*}
\xi_{0} := \int^0_{-\infty} e^{\theta s}\,\ud \hat B_s
\end{equation*}
 leads to a unique (in the sense of finite dimensional distributions) stationary Gaussian process $U^{(H)}$ of the form 

\begin{equation}\label{U^{(H)}}
U^{(H)}_t = \int^t_{-\infty} e^{-\theta (t-s)}\,\ud \hat B_s.
\end{equation}

\begin{defn}\cite{k-s}
We call the process $U^{(H,\xi_0)}$ given by $(\ref{generalfOU1})$ a fractional Ornstein-Uhlenbeck process of the first kind with initial value $\xi_0$. 
The process $U^{(H)}$ defined in $(\ref{U^{(H)}})$ is called 
stationary fractional Ornstein-Uhlenbeck process of the first kind.
\end{defn}

\begin{rmk}{ \rm
 It is shown in \cite{C-K-M} that the covariance function of the stationary process $U^{(H)}$ decays like a power function. Hence it is ergodic and 
 for $H \in (\frac{1}{2},1)$ it 
exhibits long range dependence.
}
\end{rmk}

\subsubsection{Fractional Ornstein-Uhlenbeck processes of the second kind}\label{fou2}
Now we define a new stationary Gaussian process $X^{(\alpha)}$ by means of Lamperti transformation of the fractional Brownian motion $B$. Precisely, we set
\begin{equation}\label{eq:lamperti}
 X_t^{(\alpha)} := e^{-\alpha t}B_{a_t},\quad t\in \R,
\end{equation}
where $\alpha>0$ and $a_t= \frac{H}{\alpha}e^{\frac{\alpha t}{H}}$. We aim to represent the process $X^{(\alpha)}$ as solution of a Langevin equation. To this end, we consider 
the process $Y^{\alpha}_t$ defined via
\begin{equation*}
 Y_t^{(\alpha)} := \int^t_0 e^{-\alpha s} \,\mathrm{d}B_{a_s}, \quad t \ge 0.
\end{equation*} 
As before the stochastic integral can be understood as pathwise Riemann-Stieltjes integral as well as Wiener integral. Using the self-similarity property of fractional Brownian motion one can see 
that (\cite[Proposition 6]{k-s}) the process $Y^{(\alpha)}$ satisfies the following scaling property

\begin{equation}\label{scaling}
 \{ Y^{(\alpha)}_{t / \alpha} \}_{t \ge 0} \stackrel{\text{f.d.d}}{=} \{ \alpha^{-H} Y^{(1)}_{t}\}_{t \ge 0},
\end{equation}
where $\stackrel{\text{f.d.d}}{=}$ stands for equality in finite dimensional distributions. Using $Y^{(\alpha)}$ the process $ X^{(\alpha)}$ can be viewed as the solution of the following Langevin equation
\begin{equation*}
 \mathrm{d}X_t^{(\alpha)} = -\alpha X_t^{(\alpha)}\,\mathrm{d}t+\mathrm{d}Y_t^{(\alpha)}
\end{equation*}
with random initial value $X_0^{(\alpha)}=B_{a_0} = B_{H/\alpha}\sim \mathcal{N}(0, (\frac{H}{\alpha})^{2H})$. Taking into account the scaling property $($\ref{scaling}$)$, we consider the following Langevin 
equation 
\begin{equation}
\label{Langevin}
 \mathrm{d}X_t = -\theta X_t\,\mathrm{d}t+\mathrm{d}Y_t^{(1)},\qquad   \theta > 0.
\end{equation}



with $Y^{(1)}$ as the driving noise. The solution of the equation $(\ref{Langevin})$ is given by 
\begin{equation}\label{generalfOU2}
X_t = e^{- \theta t} \left( X_0 + \int_{0}^{t} e^{\theta s} \,\mathrm{d} Y^{(1)}_s \right) = e^{- \theta t} \left( X_0 + \int_{0}^{t} e^{(\theta -1)s} \,\mathrm{d} B_{a_s} \right)
\end{equation}
with $\alpha=1$ in $a_t$. Notice that the stochastic integral can be understood as pathwise Riemann-Stieltjes integral. The special selection $X_0 = \int^0_{-\infty} e^{(\theta-1) s} \,\mathrm{d}B_{a_s}$ for the 
initial value $X_0$ leads to the following unique stationary Gaussian process 
\begin{equation}\label{U}
 U_t= e^{-\theta t}\int^t_{-\infty} e^{(\theta-1) s} \,\mathrm{d}B_{a_s}.
\end{equation}

\begin{defn}\cite{k-s}
We call the process $X$ given by $(\ref{generalfOU2})$ a fractional Ornstein-Uhlenbeck process of the second kind with initial value $X_0$. The process $U$ defined in $(\ref{U})$ is called the stationary 
fractional Ornstein-Uhlenbeck process of the second kind.
\end{defn}

For the rest of the paper we assume $H> \frac{1}{2}$. In the general solution $(\ref{generalfOU2})$, take the initial value $X_0=0$. Then the corresponding fractional Ornstein-Uhlenbeck process of the second kind 
takes the form
\begin{equation}
\label{FOU2}
X_t  = e^{- \theta t}  \int_{0}^{t} e^{(\theta -1)s} \,\mathrm{d} B_{a_s}.
\end{equation}
Notice that we have the useful relation
\begin{equation}
\label{UX_connection}
U_t = X_t + e^{-\theta t}\xi, \quad \xi = \int_{-\infty}^0 e^{(\theta -1)s}\ud B_{a_s}.
\end{equation}

We start with a series of known results, but required for our purposes, on fractional Ornstein-Uhlenbeck processes of the second kind.

\begin{prop}\cite{a-m}
\label{azmoodeh_lemma}
 Denote $\tilde B_t= B_{t+H} - B_{H}$ the shifted fractional Brownian motion. Let $X$ be the fractional Ornstein-Uhlenbeck process of the second kind given by $(\ref{FOU2})$. Then there exists a Volterra kernel 
 $\tilde{L}$ such that 
 
 \begin{equation}
\{X_t\}_{t \in [0,T]} \stackrel{\text{f.d.d}}{=} \{ \int_0^t e^{-\theta(t-s)}\ud \tilde{G}_s\}_{t \in [0,T]} 
\end{equation}
where the Gaussian process $\tilde{G}$ is given by
$$
\tilde{G}_t = \int_0^t \left( K_H(t,s) + \tilde{L}(t,s) \right) \ud\tilde{W}_s
$$
and the Brownian motion $\tilde{W}$ is related to the shifted fractional Brownian motion $\tilde{B}$ by the inverse formula $(\ref{BW})$.
\end{prop}

\begin{rmk}\label{inner_hilbert_G}{ \rm
Notice that by a direct computation and applying Lemma 4.3 of \cite{a-m}, the inner product of the Hilbert space $\tilde{\mathcal{H}}$ generated by the covariance function of the Gaussian process $\tilde{G}$ is 
given by
\begin{equation*}
 \langle \varphi,\psi \rangle_{\tilde{\mathcal{H}}} = \alpha_H H^{2H-2}\int_0^T\int_0^T \varphi(u)\psi(v)e^{(u+v)\left(\frac{1}{H}-1\right)}\left|e^{\frac{u}{H}} - e^{\frac{v}{H}} \right|^{2H-2}\ud v\ud u
\end{equation*}
where $\varphi, \psi \in \tilde{\mathcal{H}}$ and $\alpha_H=H(2H-1)$.
}
\end{rmk}


The following lemma plays an essential role in the paper. More precisely, we use this lemma to construct our 
estimators for drift parameter. $B(x,y)$ stands for the complete Beta function with parameters $x$ and $y$.

\begin{prop}\cite{a-m}
\label{lma:limit}
Let $X$ be the fractional Ornstein-Uhlenbeck process of the second kind given by $(\ref{FOU2})$. Then as $T \to \infty$, we have
\begin{equation*}
\label{lemma_as_convergence}
\frac{1}{T}\int_0^T X_t^2\ud t\rightarrow \Psi(\theta)
\end{equation*}
almost surely and in $L^2(\Omega)$, where 
\begin{equation}
\label{limit_function}
\Psi(\theta) = \frac{(2H-1)H^{2H}}{\theta}B((\theta - 1)H + 1, 2H-1).
\end{equation}
\end{prop}

\begin{prop}\cite{k-s}\label{covariance_fOU2}
The covariance function $c$ of the stationary process $U$ decays exponentially and hence exhibits short range dependence. More precisely 
\begin{equation*}
c(t):= \mathbb{E}(U_t U_0)= O \left( \exp \Big( - \min \{ \theta,\frac{1-H}{H}\}t \Big) \right), \quad \text{as } t \to \infty.
\end{equation*}
\end{prop}

Let $v_{U}$ be the variogram of the stationary process $U$, i.e.
\begin{equation*}
 v_{U}(t):= \frac{1}{2} \E \left( U_{t+s} - U_{s} \right)^2 = c(0) - c(t).
\end{equation*}
The following lemma tells us the behavior of the variogram function $v_{U}$ near zero. We will use this lemma in section \ref{estimationH}. For functions $f$ and $g$, the notation $f(t) \sim g(t)$ as 
$t \to 0$ means that $f(t) = g(t) + r(t)$ where $r(t)=o(g(t))$ as $t \to 0$.

\begin{lma}
\label{lma:variogram}
The variogram function $v_{U}$ satisfies in
\begin{equation*}
 v_{U}(t) \sim  H t^{2H} \quad \text{as} \ t \to 0^+.
\end{equation*}
\end{lma}

\begin{proof}
Due to \cite[Proposition 3.11]{k-s}, there exists a constant $C(H,\theta)= H(2H-1) H^{2H(1 - \theta)}$ such that
\begin{equation*}
c(t)= C(H,\theta) e^{- \theta t} \left( \int^{a_{t}}_0 \int^{a_{0}}_0 (xy)^{(\theta - 1)H} \vert x-y \vert^{2H-2} \ud x \ud y \right).
\end{equation*}
Denote the term inside parentheses by $\Phi(t)$. Then with some direct computations, one can see that
\begin{equation*}
  \Phi(t)=
   \frac{a_{0}^{2\theta H}}{ \theta H} B((\theta - 1)H +1, 2H-1)+ \frac{1}{2\theta H} ( a_{t}^{2 \theta H} -  a_{0}^{2 \theta H} ) \int_{0}^{\frac{a_0}{a_t}} z^{(\theta - 1)H} (1 - z)^{2H -2} \ud z.
\end{equation*}
Therefore 
\begin{equation}\label{eq:var2}
\begin{split}
c(t)&= \frac{(2H-1)H^{2H}}{\theta} B((\theta - 1)H +1, 2H-1) e^{- \theta t} \\
&+ \frac{(2H-1)H^{2H}}{2 \theta} (e^{\theta t} - e^{- \theta t} ) \int_{0}^{\frac{a_0}{a_t}} z^{(\theta - 1)H} (1 - z)^{2H -2} \ud z\\
& = c(0) - (2H-1)H^{2H} \times t \times \int_{\frac{a_0}{a_t}}^{1} z^{(\theta - 1)H} (1 - z)^{2H -2} \ud z + r(t),
\end{split}
\end{equation}
where $r(t)=o(t^{2H})$ as $t \to 0^+$. Now, using the mean value Theorem, we infer that as $t \to 0^+$ we have

\begin{equation}\label{eq:var3}
\int_{\frac{a_0}{a_t}}^{1}  z^{(\theta - 1)H}  (1 - z)^{2H -2} \ud z  \sim  \frac{H H^{-2H}}{2H-1} t^{2H-1}.
\end{equation}
Now with substituting $(\ref{eq:var3})$ into $(\ref{eq:var2})$, we obtain the claim.
\end{proof}

The next lemma studies regularity of sample paths of the fractional Ornstein-Uhlenbeck process of the second kind $X$. Usually H\"older constants are almost surely finite random variables and depend on bounded 
time intervals where the process is considered. The next lemma gives more probabilistic information on H\"older constants. 

\begin{lma}
\label{basic_for_X}
Let $X$ be the fractional Ornstein-Uhlenbeck process of the second kind given by $(\ref{FOU2})$. Then for every interval $[S,T]$ and 
every $0 < \epsilon < H$, there exist random variables $Y_1=Y_1(H,\theta)$, 
$Y_2 = Y_2(H,\theta,[S,T])$, $Y_3 = Y_3(H,\theta,[S,T])$, and $Y_4 = Y_4(H,\epsilon,[S,T])$ such that for all $s,t \in [S,T]$

\begin{equation*}
|X_t - X_s| \leq \left( Y_1+Y_2+Y_3 \right)|t-s| + Y_4|t-s|^{H-\epsilon}
\end{equation*}
almost surely. Moreover, 

\begin{itemize}

 \item [(i)]
 $Y_1<\infty$ \ almost surely,
 \item[(ii)]
 $Y_k (H,\theta,[S,T]) \law Y_k(H,\theta,[0,T-S]), \quad k= 2,3,$
 \item[(iii)]
 $Y_4(H,\epsilon,[S,T]) \law Y_4(H,\epsilon,[0,T-S]).$
 \end{itemize}

Furthermore, all moments of random variables $Y_2$, $Y_3$ and $Y_4$ are finite, and $Y_2(H,\theta,[0,T])$, $Y_3(H,\theta,[0,T])$ and 
$Y_4(H,\epsilon,[0,T])$ are increasing in $T$.
\end{lma}
\begin{proof}
Assume that $s<t$. By change of variables formula we obtain
\begin{equation*}
X_t = e^{-t}B_{a_t}-e^{-\theta t}B_{a_0} - Z_t,
\end{equation*}
where 
$$
Z_t = e^{-\theta t}\int_0^t B_{a_u}e^{(\theta -1)u}\ud u.
$$
Therefore
\begin{equation*}
\begin{split}
|X_t - X_s| &\leq |B_{a_0}||e^{-\theta t} - e^{-\theta s}| + e^{-t}|B_{a_t} - B_{a_s}| + |B_{a_s}||e^{-t}-e^{-s}| \\
&+ \left|e^{-\theta t}\int_0^t B_{a_u}e^{(\theta -1)u}\ud u - e^{-\theta s}\int_0^s B_{a_u}e^{(\theta -1)u}\ud u\right| \\
&= I_1 + I_2 + I_3 + I_4.
\end{split}
\end{equation*}
For the term $I_1$, we obtain 
\begin{equation*}
I_1 \leq \theta|B_{a_0}||t-s|
\end{equation*}
where $\theta|B_{a_0}|$ is almost surely finite random variable. For the term $I_3$, we get
\begin{equation*}
I_3 \leq \sup_{u\in[S,T]}e^{-u}|B_{a_u}||t-s|.
\end{equation*}
Note that $Z$ is a differentiable process. Hence for the term $I_4$, we get
\begin{equation*}
I_4 \leq \left[\theta\sup_{u\in[S,T]}|Z_u| + \sup_{u\in[S,T]}e^{-u}|B_{a_u}|\right]|t-s|.
\end{equation*}
Moreover, by using (\ref{UX_connection}), we have
\begin{equation*}
|X_t| \leq |U_t| + |\xi|.
\end{equation*}
As a result, we obtain
\begin{equation*}
|Z_u| \leq |U_u|+|\xi| + |B_{a_0}| + |e^{-u}B_{a_u}|.
\end{equation*}
This implies that
\begin{equation*}
I_4 \leq \left[\theta\sup_{u\in[S,T]}|U_u| +\theta|\xi|+\theta |B_{a_0}| +(\theta+1)\sup_{u\in[S,T]}e^{-u}|B_{a_u}|\right]|t-s|.
\end{equation*}
Collecting the estimates for $I_1$, $I_3$ and $I_4$, we obtain
\begin{equation*}
\begin{split}
I_1 + I_3 + I_4 & \le  \Big[2\theta|B_{a_0}| + \theta|\xi|\Big]|t-s|\\
&+ \left[\theta \sup_{u\in[S,T]}|U_u|+(\theta+2)\sup_{u\in[S,T]}e^{-u}|B_{a_u}|\right]|t-s|.
\end{split}
\end{equation*}

Put

$$Y_1= 2\theta|B_{a_0}| + \theta|\xi|, \quad Y_2(H,\theta,[S,T]) := \theta \sup_{u\in[S,T]}|U_u| $$

and finally
$$
Y_3(H,\theta,[S,T]) := (\theta+2)\sup_{u\in[S,T]}e^{-u}|B_{a_u}|.
$$
Obviously for the random variable $Y_1$ the property $(i)$ fulfills. Notice that $U_t$ and $e^{-u}B_{a_t}$ are continuous, stationary Gaussian processes. Hence the property $(ii)$ follows. Moreover, 
all moments of supremum of a continuous 
Gaussian process on a compact interval are finite (see \cite{l} for details on supremum of continuous Gaussian process). So it remains to consider the term $I_2$. By H\"{o}lder continuity of the 
sample paths of fractional Brownian motion we 
obtain
\begin{equation*}
\begin{split}
I_2 &\leq e^{-t}C(\omega,H,\epsilon,[S,T])|a_t - a_s|^{H-\epsilon}\\
&\leq C(\omega,H,\epsilon,[S,T])|t-s|^{H-\epsilon}.
\end{split}
\end{equation*}
To conclude, we obtain (see \cite{n-r} and remark below) that the random variable $C(\omega,H,\epsilon,[S,T])$ has all the moments and 
$C(\omega,H,\epsilon,[S,T]) \law C(\omega,H,\epsilon,[0,T-S])$.  Now it is enough to take $Y_4 = C(\omega,H,\epsilon,[S,T])$.
\end{proof}
\begin{rmk}\label{rmk:y4}{ \rm
The exact form of the random variable $C(\omega,H,\epsilon,[0,T])$ is given by 
$$
C(\omega,H,\epsilon,[0,T]) = C_{H,\epsilon}T^{H-\epsilon}\left(\int_0^T\int_0^T \frac{|B_t-B_s|^{\frac{2}{\epsilon}}}{|t-s|^{\frac{2H}{\epsilon}}}\ud t\ud s\right)^{\frac{\epsilon}{2}},
$$
where $C_{H,\epsilon}$ is a constant. Also, for all $p \ge 1$ and some constant $c_{\epsilon,p}$, we have $\E C(\omega,H,\epsilon,[0,T])^p\leq c_{\epsilon,p}T^{\epsilon p}$.
}
\end{rmk}

\section{Estimation of drift parameter when $H$ is known}\label{Hknown}
We start with the fact that the function $\Psi$ is invertible. This fact allows us to construct an estimator for the drift parameter $\theta$. 
\begin{lma}
\label{lma:invertible}
The function $\Psi:\R_+\rightarrow \R_+$ given by (\ref{limit_function}) is bijective, and hence invertible. 
\end{lma}
\begin{proof}

It is straightforward to see that $\Psi$ is surjective. Hence the claim follows because for any fixed parameter $y>0$, the complete Beta function $B(x,y)$ is decreasing in the variable $x$.

\end{proof}

We continue with the following central limit theorem. 
\begin{thm}
\label{thm:limit2}
Let $X$ be the fractional Ornstein-Uhlenbeck process of the second kind given by $(\ref{FOU2})$. Then as $T$ tends to infinity, we have
\begin{equation*}
\sqrt{T}\left( \frac{1}{T}\int_0^T X_t ^2\ud t - \Psi(\theta)\right)\overset{\text{law}}\longrightarrow \mathcal{N}(0,\sigma^2)
\end{equation*}
where the variance $\sigma^2$ is given by
\begin{equation}
\label{variance}
\begin{split}
\sigma^2 &= \frac{2\alpha_H^2H^{4H-4}}{\theta^2}\int_{[ 0,\infty)^{3}} \Big[ e^{-\theta x-\theta|y-z|} e^{\left(1-\frac{1}{H}\right)(x+y+z)}\\
& \hskip2cm \times \left(1-e^{-\frac{y}{H}}\right)^{2H-2} \left|e^{-\frac{x}{H}}-e^{-\frac{z}{H}}\right|^{2H-2} \Big] \ud z\ud x\ud y.
\end{split}
\end{equation}
\end{thm}
The proof relies on two lemmas proved in the appendix where we also show that $\sigma^2<\infty$. 
The variance $\sigma^2$ is given as iterated integral over $[0,\infty)^3$ and the given equation is probably the most compact form.

\begin{proof}[Proof of Theorem \ref{thm:limit2}]
For further use, put
\begin{equation}
\label{def:F_t}
F_T = \frac{1}{\sqrt{T}} I_2^{\tilde{G}}(\tilde{g}),
\end{equation}
where the symmetric function $\tilde{g}$ of two variables is given by 
\begin{equation*}
\tilde{g}(x,y)=\frac{1}{2\theta}\left[e^{-\theta|x-y|}-e^{-\theta(2T-x-y)}\right].
\end{equation*}
The notation $I_2^{\tilde{G}}$ refers to multiple Wiener integral with respect to $\tilde{G}$ introduced in Subsection \ref{Malliavin}. Next by Proposition \ref{azmoodeh_lemma}, we have 
\begin{equation*}
X_t \law I_1^{\tilde{G}}\left(h(t,\cdot)\right),\quad h(t,s)= e^{-\theta(t-s)}\textbf{1}_{s\leq t}.
\end{equation*}
Using product formula for multiple Wiener integrals and Fubini's theorem we infer that
\begin{equation*}
\begin{split}
\frac{1}{T}\int_0^T X_t^2\ud t &\law \frac{1}{T}\int_0^T \hnorm{h(t,\cdot)}^2\ud t + 
\frac{1}{T}I_2^{\tilde{G}}\left(\int_0^T \left(h(t,\cdot)\tilde{\otimes}h(t,\cdot)\right)\ud t\right)\\
&= \frac{1}{T} \int_0^T\E X_t^2\ud t + \frac{1}{T} I_2^{\tilde{G}}\left( \tilde{g} \right).
\end{split}
\end{equation*}
We get
\begin{equation}
\label{proof_step1}
\sqrt{T}\left(\frac{1}{T}\int_0^T X_t^2\ud t - \Psi(\theta)\right)
\law  \sqrt{T}\left(\frac{1}{T}\int_0^T \E X_t^2\ud t - \Psi(\theta)\right) + F_T.
\end{equation}
Next we note that $($see \cite[Lemma 3.4]{a-m}$)$
\begin{equation*}
\Psi(\theta) = \E U_0^2 = \frac{1}{T}\int_0^T\E U_0^2 \ud t.
\end{equation*}
Hence we have
\begin{equation*}
\begin{split}
\frac{1}{T}\int_0^T \E X_t^2\ud t - \Psi(\theta)&= \frac{1}{T}\int_0^T \Big( \E X_t^2-\E U_0^2 \Big) \ud t\\
&=\E U_0^2 \ \frac{1}{T} \int^T_0 e^{-2\theta t} \ud t - \frac{2}{T} \int^T_0 e^{- \theta t} \E(U_t U_0) \ud t.
\end{split}
\end{equation*}
Thus, by Proposition \ref{covariance_fOU2}, we obtain that as $T$ tends to infinity
\begin{equation}
\label{proof_step2}
\sqrt{T}\left(\frac{1}{T}\int_0^T \E X_t^2\ud t - \Psi(\theta)\right) \rightarrow 0.
\end{equation}
Therefore it suffices to show that as $T$ tends to infinity
\begin{equation*}
F_T \stackrel{\text{law}}{\to} \mathcal{N}(0,\sigma^2).
\end{equation*}
Now, by Lemmas \ref{lma:proof_step3} and \ref{lma:proof_step4} presented in the Appendix A, as $T$ tends to infinity, we have
\begin{equation*}
\hnorm{D_sF_T}^2 \overset{L^2 (\Omega)}\longrightarrow 2\sigma^2 \quad \text{and} \quad \E(F_T^2) = \frac{2}{T}\Vert\tilde{g}\Vert_{\tilde{\mathcal{H}}^{\otimes 2}}^2 \longrightarrow \sigma^2.
\end{equation*}
So the result follows by applying Proposition \ref{CLT}. 
\end{proof}

Now we are ready to state the main result of this section. 
\begin{thm}
\label{main_theorem}
Assume we observe the fractional Ornstein-Uhlenbeck process of the second kind $X$ given by $(\ref{FOU2})$ at discrete time points $\{t_k = k\Delta_N, k=0,1,\ldots,N\}$. Let $T_N = N\Delta_N$. 
Assume we have $\Delta_N \to 0, \ T_N\rightarrow\infty$ and $N\Delta_N^{2}\rightarrow 0$ as $N$ tends to infinity. Put
\begin{equation}
\label{mu_est}
\widehat{\mu}_{2,N} = \frac{1}{T_N}\sum_{k=1}^N X_{t_k}^2\Delta t_k \quad \text{and} \quad \widehat{\theta}_N := \Psi^{-1}\left(\widehat{\mu}_{2,N}\right),
\end{equation}
 where $\Psi^{-1}$ is the inverse of the function $\Psi$ given by (\ref{limit_function}). Then $\widehat{\theta}$ is a strongly consistent estimator of the drift parameter $\theta$ in the sense that as $N$ tends 
 to infinity, we have
\begin{equation}
\label{main_convergence}
\widehat{\theta}_N   \longrightarrow \theta
\end{equation} 
almost surely. Moreover, as $N$ tends to infinity, we have
\begin{equation}
\label{main_convergence_d}
\sqrt{T_N}(\widehat{\theta}_N - \theta)\overset{\text{law}} \longrightarrow \mathcal{N}(0,\sigma_{\theta}^2)
\end{equation}
where
\begin{equation}
\label{variance_theta}
\sigma_{\theta}^2 = \frac{\sigma^2}{[\Psi'(\theta)]^2}
\end{equation}
and $\sigma^2$ is given by (\ref{variance}).
\end{thm}

\begin{proof}
Applying Lemma \ref{basic_for_X}, for any $\epsilon \in (0,H)$, we obtain
\begin{equation*}
\begin{split}
\sqrt{T_N} \Big\vert \widehat{\mu}_{2,N}  -  \frac{1}{T_N} &\int_0^{T_N}X_t^2\ud t \Big\vert =\frac{1}{\sqrt{T_N}}\left|\sum_{k=1}^N \int_{t_{k-1}}^{t_k}(X_{t_k}^2 - X_t^2)\ud t\right| \\
&\leq \frac{2}{\sqrt{T_N}}\left(\sum_{k=1}^N \sup_{u\in[t_{k-1},t_k]}|X_u|\int_{t_{k-1}}^{t_k}|X_{t_k} - X_t|\ud t\right)\\
&\leq \frac{2Y_1(H,\theta)}{\sqrt{T_N}}\left(\sum_{k=1}^N \sup_{u\in[t_{k-1},t_k]}|X_u|\int_{t_{k-1}}^{t_k}(t_k - t)\ud t\right)\\
&+\frac{2}{\sqrt{T_N}}\left(\sum_{k=1}^N \sup_{u\in[t_{k-1},t_k]}|X_u|      Y_2(H,\theta,[t_{k-1},t_k])\int_{t_{k-1}}^{t_k}(t_k - t)\ud t\right)\\
&+\frac{2}{\sqrt{T_N}}\left(\sum_{k=1}^N \sup_{u\in[t_{k-1},t_k]}|X_u|Y_3(H,\theta,[t_{k-1},t_k])\int_{t_{k-1}}^{t_k}(t_k - t)\ud t\right)\\
&+\frac{2}{\sqrt{T_N}}\left(\sum_{k=1}^N \sup_{u\in[t_{k-1},t_k]}|X_u|Y_4(H,\epsilon,[t_{k-1},t_k])\int_{t_{k-1}}^{t_k}(t_k - t)^{H-\epsilon}\ud t\right)\\
&=: I_1+I_2 + I_3+I_4.
\end{split}
\end{equation*}
We begin with last term $I_4$. Clearly we have
\begin{equation*}
\begin{split}
\sum_{k=1}^N \sup_{u\in[t_{k-1},t_k]}|X_u| \, Y_4(H,\epsilon,[t_{k-1},t_k]) \leq N \sup_{u\in[0,T_N]}|X_u| \, Y_4(H,\epsilon,[0,T_N]).
\end{split}
\end{equation*}
By Remark \ref{rmk:y4}, we have $\E Y_4(H,\epsilon,[0,T_N])^p \leq CT_N^{\epsilon p}$ for any $p \ge 1$. Hence using Markov's inequality, we obtain for every $\delta>0$ that
$$
\P\left(N^{-\gamma}Y_4(H,\epsilon,[0,T_N]) > \delta\right) \leq \frac{C^pT_N^{\epsilon p}}{N^{\gamma p}\delta^p}.
$$
Now by choosing $\epsilon<\gamma$ and $p$ large enough we obtain
$$
\sum_{N=1}^\infty \P\left(N^{-\gamma}Y_4(H,\epsilon,[0,T_N]) > \delta\right) < \infty.
$$
Consequently, Borel-Cantelli Lemma implies that 
$$
N^{-\gamma}Y_4(H,\epsilon,[0,T_N]) \rightarrow 0
$$
almost surely for any $\gamma >\epsilon$. Similarly, we obtain 
$$
N^{-\gamma}\sup_{u\in[0,T_N]}|X_u| \rightarrow 0
$$
almost surely for any $\gamma>0$. Consequently, we get
$$
\frac{1}{N^{1+2\gamma}}\sum_{k=1}^N \sup_{u\in[t_{k-1},t_k]}|X_u|Y_4(H,\epsilon,[t_{k-1},t_k])
\longrightarrow 0
$$
almost surely for any $\gamma > \epsilon$. Note also that by choosing $\epsilon>0$ small enough we can choose $\gamma$ in such way that $1+2\epsilon<1+2\gamma < \frac{3}{4} + \frac{H-\epsilon}{2}$. In particular, 
this is possible if $\epsilon < \min\left\{ H-\frac{1}{2},\frac{H}{5}\right\}$. With this choice we have
\begin{equation*}
\begin{split}
I_4& \leq \frac{2}{H-\epsilon+1}\sqrt{T_N}\Delta_N^{H-\epsilon}\frac{1}{N}\sum_{k=1}^N \sup_{u\in[t_{k-1},t_k]}|X_u|Y_4(H,\epsilon,[t_{k-1},t_k])\\
&= \frac{2}{H-\epsilon+1}\sqrt{T_N}\Delta_N^{H-\epsilon} N^{2\gamma} \frac{1}{N^{1+2\gamma}}\sum_{k=1}^N \sup_{u\in[t_{k-1},t_k]}|X_u|Y_4(H,\epsilon,[t_{k-1},t_k])\\
&\longrightarrow 0
\end{split}
\end{equation*}
almost surely, because the condition $N\Delta_N^2\rightarrow 0$ and our choice of $\gamma$ implies that 
\begin{equation*}
\begin{split}
\sqrt{T_N}\Delta_N^{H-\epsilon}N^{2\gamma} & =\left(N\Delta_N^{\frac{2H+1-2\epsilon}{1+4\gamma}}\right)^{2\gamma+\frac{1}{2}} \leq \left(N\Delta_N^2\right)^{2\gamma+\frac{1}{2}} \rightarrow 0.
\end{split}
\end{equation*}
Treating $I_1$, $I_2$, and $I_3$ in a similar way, we deduce that
\begin{equation}
\label{same_limit}
\sqrt{T_N}\left|\widehat{\mu}_{2,N} - \frac{1}{T_N}\int_0^{T_N}X_t^2\ud t\right|\rightarrow 0
\end{equation}
almost surely. Moreover, we have convergence (\ref{main_convergence}) by Lemma \ref{lma:limit}. To conclude the proof, we set $\mu = \Psi(\theta)$ and use Taylor's theorem to obtain
\begin{equation*}
\begin{split}
\sqrt{T_N}\left(\widehat{\theta}_N - \theta\right) &=\frac{\ud}{\ud \mu} \Psi^{-1} (\mu)\sqrt{T_N}\left(\widehat{\mu}_{2,N} - \Psi(\theta)\right) \\
&+ R_1(\widehat{\mu}_{2,N})\sqrt{T_N}\left(\widehat{\mu}_{2,N} - \Psi(\theta)\right)\\
&=\frac{\ud }{\ud \mu} \Psi^{-1} (\mu)\sqrt{T_N}\left(\frac{1}{T_N}\int_0^{T_N}X_t^2\ud t - \Psi(\theta)\right) \\
&+\frac{\ud }{\ud \mu} \Psi^{-1} (\mu)\sqrt{T_N}\left(\widehat{\mu}_{2,N} - \frac{1}{T_N}\int_0^{T_N}X_t^2\ud t\right) \\
&+ R_1(\widehat{\mu}_{2,N})\sqrt{T_N}\left(\widehat{\mu}_{2,N} - \Psi(\theta)\right)
\end{split}
\end{equation*}
for some reminder function $R_1(x)$ such that $R_1(x)\rightarrow 0$ when $x\rightarrow \Psi(\theta)$. Now, by continuity of $\frac{\ud}{\ud \mu} \Psi^{-1}$ and $\Psi^{-1}$, we have that $R_1$ is also continuous. Hence the result follows by 
using (\ref{same_limit}), Theorem \ref{thm:limit2}, Slutsky's theorem and the fact that
$$
\frac{\ud}{\ud \mu} \Psi^{-1}(\mu) = \frac{1}{\Psi'(\theta)}.
$$
\end{proof}

\begin{rmk}{ \rm
We remark that it is straightforward to construct strongly consistent estimator without the mesh restriction $\Delta_N \rightarrow 0$. However, in order to obtain central limit theorem using 
Theorem \ref{thm:limit2}, one need to pose the condition $\Delta_N \rightarrow 0$ to get the convergence
$$
\sqrt{T_N}\left|\widehat{\mu}_{2,N} - \frac{1}{T_N}\int_0^{T_N}X_t^2\ud t\right|\rightarrow 0.
$$
}
\end{rmk}

\begin{rmk}{ \rm
Notice that we obtained a consistent estimator which depends on the inverse of the function $\Psi$. Despite we proved that such inverse exists, but up to our knowledge there exists not an explicit formula for the 
inverse. Hence the inverse can be computed numerically.
}
\end{rmk}

\begin{rmk}{ \rm
Theorem \ref{main_theorem} imposes different conditions on the mesh $\Delta_N$ to obtain strong consistency of the estimator $\widehat{\theta}$. 
One possible choice for the mesh satisfying such conditions is $\Delta_N = \frac{\log N}{N}$.
}
\end{rmk}

\begin{rmk}{ \rm
Notice that we obtained strong consistency of the estimator $\widehat{\theta}$ without assuming uniform discretization of the partitions. The uniform discretization will play a role in estimating Hurst 
parameter $H$ and moreover for analysis of the estimator $\widetilde{\theta}$ (see Sections \ref{estimationH} and \ref{unknownH}).
}
\end{rmk}

\section{Estimation of Hurst parameter $H$}\label{estimationH}
There are different approaches to estimate Hurst parameter $H$ of fractional processes. Among all, here we consider an approach which is based on filtering of 
discrete observations of a sample path of process. For more details we refer to \cite{I-L,co}.\\

Let $\textbf{a}=(a_0,a_1, \cdots,a_L) \in \R^{L+1}$ be a filter of length $L+1, L \in \N$, and of order $p\ge 1$, i.e. for all indices $0 \le q < p$,
\begin{equation*}
 \sum_{j=0}^{L} a_j j^{q}=0 \quad \text{and} \quad \sum_{j=0}^{L} a_j j^{p}\neq 0.
\end{equation*}
We define the dilated filter $\textbf{a}^2$ associated to the filter $\textbf{a}$ by 
\begin{equation*}
{a}^2_{k} = \begin{cases}
& a_{k'},\quad k=2k'\\
& 0, \quad \text{otherwise}
\end{cases}
\end{equation*}
for $0\leq k\leq 2L$. Assume that we observe the process $X$ given by $(\ref{FOU2})$ at discrete time points $\{t_k = k\Delta_N, k=1,\ldots,N\}$ such that the mesh $\Delta_N \to 0$ as $N$ tends to infinity. We denote the generalized 
quadratic variations associated to filter $\textbf{a}$ by 
\begin{equation*}
V_{N,\textbf{a}} = \frac{1}{N} \sum_{i=0}^{N-L}\left(\sum_{j=0}^L a_j X_{(i+j)\Delta_N}\right)^2.
\end{equation*}
We consider the estimator $\widehat{H}_N$ given by
\begin{equation}
\label{h_est}
\widehat{H}_N = \frac{1}{2}\log_2 \frac{V_{N,\textbf{a}^2}}{V_{N,\textbf{a}}}.
\end{equation}

\textbf{Assumption (A)}:\\

We say the filter $\textbf{a}$ of the length $L+1$ and order $p$ satisfies in the assumption \textbf{(A)} if for any real number $r$ such that $0 < r < 2p$ and $r$ is not an even integer, the following property 
holds:
\begin{equation*}
\sum_{i=0}^{L} \sum_{j=0}^{L} a_i a_j \vert i - j \vert^{r} \neq 0.
\end{equation*}

\begin{exm}
A typical example of a filter with finite order satisfying the assumption \textbf{(A)} is $\textbf{a} = (1,-2,1)$ with order $p=2$.
\end{exm}

\begin{thm}\label{thm:Hestimation}
\label{thm:h_est}
Let $\textbf{a}$ be a filter of the order $p\geq 2$ satisfying in assumption \textbf{(A)}. Let the mesh $\Delta_N = N^{- \alpha}$ for some 
$\alpha \in (\frac{1}{2},\frac{1}{4H-2})$. Then
\begin{equation*}
\widehat{H}_N \longrightarrow H
\end{equation*}
almost surely as $N$ tends to infinity. Moreover, as $N$ tends to infinity, we have 
\begin{equation*}
\sqrt{N}(\widehat{H}_N-H))\overset{\text{law}}\longrightarrow \mathcal{N}(0,\Gamma(H,\theta,\textbf{a}))
\end{equation*}
where the variance $\Gamma$ depends on $H$, $\theta$ and the filter $\textbf{a}$ and is explicitly computed in \cite{co} and also given in \cite{b-i}.
\end{thm}

\begin{rmk} { \rm
It is worth to mention that when $H < \frac{3}{4}$, it is not necessary to assume that the observation window $T_N= N \Delta_N$ tends to infinity, whereas 
 when $H \ge \frac{3}{4}$, we have to have that $T_N$ tends to infinity, see \cite{I-L}. Notice that $H \ge \frac{3}{4}$ if and only if $\frac{1}{4H-2} \le 1$.
 }
\end{rmk}

\begin{proof}[Proof of Theorem \ref{thm:Hestimation}]
Let $v_U$ denote the variogram of the process $U$. By Lemma \ref{lma:variogram}, we have
\begin{equation*}
v_U (t) = Ht^{2H} + r(t)
\end{equation*}
as $t\rightarrow 0^+$ where  $r(t)=o(t^{2H})$. Moreover, we have that $r(t)$ is differentiable and direct calculations show that for $\epsilon \in (0,1)$
\begin{equation*}
r^{(4)}(t) \leq G|t|^{2H+1-\epsilon-4}.
\end{equation*}
Hence the claim follows by following the proof in \cite{b-i} for the fractional Ornstein-Uhlenbeck process of the first kind and applying the results of Istas \& Lang \cite[Theorem 3]{I-L}. To conclude, we note that the 
given variance is also computed in \cite[page $223$]{co}. 

\end{proof}

\section{Estimation of drift parameter when $H$ is unknown}\label{unknownH}
In this section, instead of $\Psi(\theta)$ we consider $\Psi(\theta,H)$ to take into account the dependence on Hurst parameter $H$. Let $\mu=\Psi(\theta,H)$. Then implicit function theorem implies that there 
exists a continuously differentiable function $g(\mu,H)$ such that
$$
g(\mu,H) = \theta
$$
where $\theta$ is the unique solution to equation $\mu=\Psi(\theta,H)$. Hence for every fixed $H$, we have
$$
\frac{\partial g}{\partial\mu}(\mu,H) = \frac{1}{\frac{\partial\Psi}{\partial\theta}(\theta,H)}.
$$
Moreover, by chain rule we obtain
$$
0=\frac{\ud}{\ud H} g(\Psi(\theta,H),H) = \frac{\partial g}{\partial H} +\frac{\partial g}{\partial \mu}\frac{\partial \mu}{\partial H}.
$$
Here $\frac{\partial g}{\partial \mu}$ and $\frac{\partial\mu}{\partial H}$ are known, and so we can compute $\frac{\partial g}{\partial H}$. 
Let $\widehat{H}_N$ be given by (\ref{h_est}) for some filter $\textbf{a}$ of order $p\geq 2$ satisfying 
assumption \textbf{(A)} and $\widehat{\mu}_{2,N}$ by (\ref{mu_est}). We consider the estimator
\begin{equation}
\label{final_estimator}
\widetilde{\theta}_N = g(\widehat{\mu}_{2,N},\widehat{H}_N).
\end{equation}

Now with all the above assumptions and notations, we have the following result.

\begin{thm}\label{thm:main2}
Assume that mesh $\Delta_N = N^{-\alpha}$ for some number $\alpha \in (\frac{1}{2},\frac{1}{4H-2}\wedge 1)$. Then 
the estimator $\widetilde{\theta}_N$ given by $(\ref{final_estimator})$ is strongly consistent, i.e. as $N$ tends to infinity, we have
\begin{equation}
\label{final_est_convergence}
\widetilde{\theta}_N \longrightarrow \theta
\end{equation}
almost surely. Moreover, as $N$ tends to infinity, the following central limit theorem 
\begin{equation}
\label{final_est_limit_d}
\sqrt{T_N}\left(\widetilde{\theta}_N - \theta\right)\overset{\text{law}} \longrightarrow\mathcal{N}(0,\sigma_{\theta}^2)
\end{equation}
holds, where the variance $\sigma^2_{\theta}$ is given by (\ref{variance_theta}).
\end{thm}

\begin{proof}
First note that
\begin{equation}
\label{aux_equation}
\begin{split}
\sqrt{T_N}\left(\widetilde{\theta}_N - \theta\right) &=\sqrt{T_N}\left(g(\widehat{\mu}_{2,N},\widehat{H}_N) - g(\widehat{\mu}_{2,N},H)\right)\\
&+\sqrt{T_N} \Big( g(\widehat{\mu}_{2,N},H) - g(\mu,H) \Big).\\
\end{split}
\end{equation}
Now convergence 
$$
\sqrt{T_N}\Big( g(\widehat{\mu}_{2,N},H) - g(\mu,H) \Big) \overset{\text{law}}\longrightarrow\mathcal{N}(0,\sigma_{\theta}^2)
$$
is in fact Theorem \ref{main_theorem}. Moreover, using Taylor's theorem, we get that
\begin{equation*}
\begin{split}
\sqrt{T_N}\Big(g(\widehat{\mu}_{2,N},\widehat{H}_N) -& g(\widehat{\mu}_{2,N},H)\Big) = 
\frac{\partial g}{\partial H}(\widehat{\mu}_{2,N},H)\sqrt{T_N}(\widehat{H}_N - H) \\
&+ \frac{\partial g}{\partial H}(\widehat{\mu}_{2,N},H)R_2(\widehat{\mu}_{2,N},\widehat{H}_N)\sqrt{T_N}(\widehat{H}_N - H)
\end{split}
\end{equation*}
for some reminder function $R_2$ which converges to zero as $(\hat{\mu}_{2,N},\hat{H}_N)\rightarrow(\mu,H)$. Therefore, by continuity and Theorem \ref{thm:h_est}, as $N$ tends to infinity, we obtain
$$
\sqrt{T_N}\left(g(\widehat{\mu}_{2,N},\widehat{H}_N) - g(\widehat{\mu}_{2,N},H)\right)\longrightarrow 0
$$
in probability. Hence, using Slutsky's theorem we obtain
$$
\sqrt{T_N}\left(\widehat{\theta}_N - \theta\right)\overset{\text{law}}\longrightarrow\mathcal{N}(0,\sigma_{\theta}^2).
$$
To conclude the proof, we obtain (\ref{final_est_convergence}) from equation (\ref{aux_equation}) 
by continuous mapping theorem.
\end{proof}


\subsection*{Acknowledgments} 
The authors thank Lasse Leskel\"{a} for useful discussions and comments. Lauri Viitasaari thanks the Finnish Doctoral Programme in Stochastics and Statistics for financial support. Azmoodeh is supported by 
research project F1R-MTH-PUL-12PAMP. The authors thank both anonymous referees for careful reading of the previous version of this paper and for their valuable comments which improved the paper.

\appendix
\section{Computations used in the paper} 
\begin{lma}
\label{lma:proof_step3}
For $F_T$ given by (\ref{def:F_t}) and the variance $\sigma^2$ given by $(\ref{variance})$, as $T$ tends to infinity, we have
\begin{equation}
\label{proof_step3}
\hnorm{D_sF_T}^2 \overset{L^2(\Omega)}\longrightarrow 2\sigma^2.
\end{equation}
\end{lma}
\begin{proof}
It is sufficient to show that as $T$ tends to infinity
\begin{equation}
\label{proof_step5}
\E\left[\hnorm{D_sF_T}^2 - \E\hnorm{D_sF_T}^2\right]^2 \rightarrow 0.
\end{equation}
Indeed, taking into account that 
$$
\lim_{T \to \infty}\E \hnorm{D_sF_T}^2 = 2 \lim_{T \to \infty}\E(F_T^2),
$$
we obtain that (\ref{proof_step5}) implies (\ref{proof_step3}).
Now we have
$$
D_s F_T = \frac{2}{\sqrt{T}}I_1^{\tilde{G}}(\tilde{g}(s,\cdot)).
$$
Hence using the Remark \ref{inner_hilbert_G}, we can write
\begin{equation*}
\begin{split}
\hnorm{D_sF_T}^2 &= \frac{4\alpha_HH^{2H-2}}{T}\int_0^T\int_0^TI_1^{\tilde{G}}(\tilde{g}(u,\cdot))I_1^{\tilde{G}}(\tilde{g}(v,\cdot))\\
& \hspace{3cm} \times e^{(u+v)\left(\frac{1}{H}-1\right)}\left|e^{\frac{u}{H}} - e^{\frac{v}{H}}\right|^{2H-2}\ud v\ud u.
\end{split}
\end{equation*}
Let now $K(u,v)$ denote the kernel associated to the space $\tilde{\mathcal{H}}$ i.e.
\begin{equation}
\label{kernel}
K(u,v) = e^{(u+v)\left(\frac{1}{H}-1\right)}\left|e^{\frac{u}{H}} - e^{\frac{v}{H}}\right|^{2H-2}.
\end{equation}
Using multiplicative formula for multiple Wiener integrals, we see that
\begin{equation*}
\begin{split}
I_1^{\tilde{G}}\left( \tilde{g}(u,\cdot) \right) & \, I_1^{\tilde{G}}\left( \tilde{g}(v,\cdot) \right)\\
&=\langle\tilde{g}(u,\cdot),\tilde{g}(v,\cdot)\rangle_{\tilde{\mathcal{H}}}+ 
I_2^{\tilde{G}}\left(\tilde{g}(u,\cdot)\tilde{\otimes} \tilde{g}(v,\cdot)\right)\\
&=: A_1(u,v) + A_2(u,v).
\end{split}
\end{equation*}
Here $A_1$ is deterministic and $A_2$ has expectation zero. Hence, in order to have (\ref{proof_step5}), we need to show that
\begin{equation}
\label{simple_conv}
\E\left[\frac{1}{T}\int_0^T\int_0^T A_2(u,v) K(u,v)\ud v\ud u\right]^2\rightarrow 0.
\end{equation}
Therefore, by applying Fubini's Theorem, it suffices to show that
\begin{equation}
\label{apu_conv}
\begin{split}
&\frac{1}{T^2}\int_{[0,T]^4} \E\left[A_2(u_1,v_1)A_2(u_2,v_2)\right]\\
&\hskip 3cm \times K(u_1,v_1)K(u_2,v_2)\ud u_1\ud v_1 \ud u_2\ud v_2 \rightarrow 0
\end{split}
\end{equation}
as $T$ tends to infinity. First we get that
\begin{equation*}
\begin{split}
&\E\left[A_2(u_1,v_1)A_2(u_2,v_2)\right]\\
&=2 \int_{[0,T]^4}\left(\tilde{g}(u_1,\cdot)\tilde{\otimes}\tilde{g}(v_1,\cdot)\right)(x_1,y_1)
\left(\tilde{g}(u_2,\cdot)\tilde{\otimes}\tilde{g}(v_2,\cdot)\right)(x_2,y_2)\\
& \hskip 3cm\times K(x_1,x_2)K(y_1,y_2)\ud x_1\ud y_1 \ud x_2\ud y_2.
\end{split}
\end{equation*}
By plugging into (\ref{apu_conv}) we obtain that it suffices to have
\begin{equation}
\label{apu_conv2}
\begin{split}
&\frac{1}{T^2}\int_{[0,T]^8} \left(\tilde{g}(u_1,\cdot)\tilde{\otimes}\tilde{g}(v_1,\cdot)\right)(x_1,y_1)
\left(\tilde{g}(u_2,\cdot)\tilde{\otimes}\tilde{g}(v_2,\cdot)\right)(x_2,y_2)\\
& \hskip 3cm\times K(x_1,x_2)K(y_1,y_2)K(u_1,v_1)K(u_2,v_2)\\
& \hskip 3cm \ud v_1 \ud u_2\ud v_2 \ud u_1\ud x_1\ud y_1 \ud x_2\ud y_2\rightarrow 0
\end{split}
\end{equation}
as $T$ tends to infinity. Here we have
\begin{equation*}
\begin{split}
\left(\tilde{g}(u,\cdot)\tilde{\otimes}\tilde{g}(v,\cdot)\right)(x,y) =\frac{1}{2}\left[\tilde{g}(u,x)\tilde{g}(v,y)+\tilde{g}(u,y)\tilde{g}(v,x)\right].
\end{split}
\end{equation*}
Note first that for every $0\leq x,y\leq T$, we have that
\begin{equation*}
e^{-\theta(2T-x-y)}\leq e^{-\theta|x-y|}.
\end{equation*}
As a consequence, we can omit the term $e^{-\theta(2T-x-y)}$ on function $\tilde{g}(x,y)$. This implies that instead of
$$
\left(\tilde{g}(u_1,\cdot)\tilde{\otimes}\tilde{g}(v_1,\cdot)\right)(x_1,y_1)
\left(\tilde{g}(u_2,\cdot)\tilde{\otimes}\tilde{g}(v_2,\cdot)\right)(x_2,y_2)
$$
it is sufficient to consider the following integrand:
\begin{equation}
\label{enough_function}
\begin{split}
&e^{-\theta|u_1-x_1|}e^{-\theta|v_1-y_1|}e^{-\theta|u_2-x_2|}e^{-\theta|v_2-y_2|}\\
&+e^{-\theta|u_1-x_1|}e^{-\theta|v_1-y_1|}e^{-\theta|u_2-y_2|}e^{-\theta|v_2-x_2|}\\
&+e^{-\theta|u_1-y_1|}e^{-\theta|v_1-x_1|}e^{-\theta|u_2-x_2|}e^{-\theta|v_2-y_2|}\\
&+e^{-\theta|u_1-y_1|}e^{-\theta|v_1-x_1|}e^{-\theta|u_2-y_2|}e^{-\theta|v_2-x_2|}.
\end{split}
\end{equation}
Next we consider the first term and show that
\begin{equation}
\label{apu_conv3}
\begin{split}
&\frac{1}{T^2}\int_{[0,T]^8} e^{-\theta|u_1-x_1|}e^{-\theta|v_1-y_1|}e^{-\theta|u_2-x_2|}e^{-\theta|v_2-y_2|}\\
& \hskip 2 cm\times K(x_1,x_2)K(y_1,y_2)K(u_1,v_1)K(u_2,v_2)\\
& \hskip 2.5cm \ud u_1\ud v_1 \ud u_2\ud v_2 \ud x_1\ud y_1 \ud x_2\ud y_2\rightarrow 0.
\end{split}
\end{equation}
In what follows $C$ is a non-important constant which may vary from line to line. First it is easy to prove that
\begin{equation}
\label{apu_easy}
\int_0^T e^{-\theta|x-y|}\ud x \leq C,
\end{equation}
where constant does not depend on $y$ or $T$. Moreover, by change of variable we obtain 
\begin{equation}
\label{apu_easy2}
\int_0^T  K(x,y)\ud x \leq 2HB(1-H,2H-1)
\end{equation}
for every $y$ and $T$. Consider now the iterated integral in (\ref{apu_conv3}). The value of the integral depends on the order
of the variables, and eight variables can be ordered in $8! = 40320$ ways. However, it is clear that 
without loss of generality we can choose the smallest variable, let's say $y_2$, and integrate
over region $\{0<y_2<u_1,u_2,v_1,v_2,x_1,x_2,y_1<T\}$. Other cases can be treated similarly with obvious changes. 
Assume now that the smallest variable is $y_2$ and denote the second smallest variable by $r_7$, i.e.
\begin{equation*}
r_7 = \min(u_1,u_2,v_1,v_2,x_1,x_2,y_1).
\end{equation*}
Integrating first with respect to $y_2$ and applying upper bound $e^{\theta y_2} \leq e^{\theta r_7}$ together with (\ref{apu_easy2}), we obtain that
\begin{equation*}
\begin{split}
&\int_{[0,T]^7}\int_0^{r_7} e^{-\theta|u_1-x_1|}e^{-\theta|v_1-y_1|}e^{-\theta|u_2-x_2|}e^{-\theta v_2 + \theta y_2}\\
&\times K(x_1,x_2)K(y_1,y_2)K(u_1,v_1)K(u_2,v_2)
\ud y_2\ud u_1\ud v_1 \ud u_2\ud v_2 \ud x_1\ud y_1 \ud x_2\\
&\leq C\int_{[0,T]^7}e^{-\theta|u_1-x_1|}e^{-\theta|v_1-y_1|}e^{-\theta|u_2-x_2|}e^{-\theta v_2 + \theta r_7}\\
&\times K(x_1,x_2)K(u_1,v_1)K(u_2,v_2)
\ud u_1\ud v_1 \ud u_2\ud v_2 \ud x_1\ud y_1 \ud x_2.
\end{split}
\end{equation*}
Next we integrate with respect to $y_1$. In the case when  $r_7=y_1$, we have
\begin{equation*}
\int_0^{r_6} e^{-\theta(v_1+v_2-2y_1)}\ud y_1 \leq Ce^{-\theta(v_1+v_2-2r_6)} \leq C,
\end{equation*}
where $r_6$ is the third smallest variable, and in the case when $r_7 \neq y_1$, we obtain by (\ref{apu_easy})
\begin{equation*}
\int_0^T e^{-\theta|v_1-y_1|}e^{-\theta v_2 + \theta r_7}\ud y_1 \leq C.
\end{equation*}
Hence we obtain upper bound
\begin{equation*}
\begin{split}
&\int_{[0,T]^7}e^{-\theta|u_1-x_1|}e^{-\theta|v_1-y_1|}e^{-\theta|u_2-x_2|}e^{-\theta v_2 + \theta r_7}\\
&\times K(x_1,x_2)K(u_1,v_1)K(u_2,v_2)
\ud u_1\ud v_1 \ud u_2\ud v_2 \ud x_1\ud y_1 \ud x_2\\
&\leq C\int_{[0,T]^6}e^{-\theta|u_1-x_1|}e^{-\theta|u_2-x_2|}\\
&\times K(x_1,x_2)K(u_1,v_1)K(u_2,v_2)
\ud u_1\ud v_1\ud v_2 \ud x_1\ud u_2 \ud x_2.
\end{split}
\end{equation*}
Next we integrate first with respect to variables $v_1$ and $v_2$ and then with respect to variables $u_1$ and $u_2$. 
Together with estimates (\ref{apu_easy}) and (\ref{apu_easy2}) this yields
\begin{equation*}
\begin{split}
&\int_{[0,T]^6}e^{-\theta|u_1-x_1|}e^{-\theta|u_2-x_2|}\\
&\times K(x_1,x_2)K(u_1,v_1)K(u_2,v_2)
\ud v_1\ud v_2\ud u_1 \ud u_2\ud x_1 \ud x_2\\
&\leq C \int_{[0,T]^4}e^{-\theta|u_1-x_1|}e^{-\theta|u_2-x_2|}K(x_1,x_2)
\ud u_1 \ud u_2\ud x_1 \ud x_2\\
&\leq C\int_{[0,T]^2}K(x_1,x_2)\ud x_1 \ud x_2\\
&\leq CT.
\end{split}
\end{equation*}
Hence we have (\ref{apu_conv3}). It remains to note that other three terms in (\ref{enough_function}) can be treated with 
the same arguments since only the ''pairing'' of variables in terms of form $e^{-\theta|x-y|}$ changes. Thus we have (\ref{apu_conv2}) 
and implications \\(\ref{apu_conv2})$\Rightarrow$(\ref{simple_conv})$\Rightarrow$(\ref{proof_step5})
$\Rightarrow$(\ref{proof_step3}) complete the proof.
\end{proof}

\begin{lma}
\label{lma:proof_step4}
For $F_T$ given by (\ref{def:F_t}), as $T$ tends to infinity, we have
\begin{equation}
\label{proof_step4}
\E[F_T^2]  \longrightarrow \sigma^2.
\end{equation}
\end{lma}
\begin{proof}
First using isometry property, we obtain
\begin{equation*}
\E[F_T^2] =\frac{2}{T} \Vert\tilde{g}\Vert_{\tilde{\mathcal{H}}^{\otimes 2}}^2 =: \frac{2I_T}{T}
\end{equation*}
where
\begin{equation*}
\begin{split}
I_T = &\alpha_H^2H^{4H-4}\int_{[0,T]^{4}} \tilde{g}(u_1,v_1)\tilde{g}(u_2,v_2) e^{\left(\frac{1}{H}-1\right)(u_1+v_1+u_2+v_2)}\\
& \times \left|e^{\frac{u_2}{H}}-e^{\frac{u_1}{H}}\right|^{2H-2}\left|e^{\frac{v_2}{H}}-e^{\frac{v_1}{H}}\right|^{2H-2}\ud u_1\ud u_2\ud v_1\ud v_2.
\end{split}
\end{equation*}
Recall that 
$$
\tilde{g}(x,y)= \frac{1}{2\theta}e^{-\theta|x-y|} - \frac{1}{2\theta}e^{-\theta(2T-x-y)}.
$$
We first show that we can omit the second term $\frac{1}{2\theta}e^{-\theta(2T-x-y)}$ in the function $\tilde{g}$. To see this, we have
\begin{equation*}
\begin{split}
\int_{[0,T]^{4}}&   e^{-\theta(2T-u_1-v_1)}\tilde{g}(u_2,v_2) e^{\left(\frac{1}{H}-1\right)(u_1+v_1+u_2+v_2)}\\
&\times\left|e^{\frac{u_2}{H}}-e^{\frac{u_1}{H}}\right|^{2H-2} \left|e^{\frac{v_2}{H}}-e^{\frac{v_1}{H}}\right|^{2H-2}\ud u_1\ud u_2\ud v_1\ud v_2\\
&\leq C(\theta) \int_{[0,T]^{4}} e^{-\theta(2T-u_1-v_1)} e^{\left(\frac{1}{H}-1\right)(u_1+v_1+u_2+v_2)}\\
&\times \left|e^{\frac{u_2}{H}}-e^{\frac{u_1}{H}}\right|^{2H-2}\left|e^{\frac{v_2}{H}}-e^{\frac{v_1}{H}}\right|^{2H-2}\ud u_1\ud u_2\ud v_1\ud v_2\\
&= C(\theta) \left[\int_0^T\int_0^T e^{-\theta (T-v) + \left(\frac{1}{H}-1\right)(v+u)}\left|e^{\frac{u}{H}}-e^{\frac{v}{H}}\right|^{2H-2}
\ud v\ud u\right]^2.
\end{split}
\end{equation*}
By change of variables $\tilde{v}=T-v$, $\tilde{u}=T-u$, and then $x=e^{-\frac{\tilde{v}}{H}}$, $y=e^{-\frac{\tilde{u}}{H}}$ this is same as
\begin{equation*}
\left[\int_{e^{-\frac{T}{H}}}^1\int_{e^{-\frac{T}{H}}}^1 x^{(\theta-1)H}y^{-H}|y-x|^{2H-2}\ud x\ud y\right]^2.
\end{equation*}
Let now $x<y$. By change of variables $z=\frac{x}{y}$, we obtain
\begin{equation*}
\begin{split}
 \int_{e^{-\frac{T}{H}}}^1\int_{e^{-\frac{T}{H}}}^y &  x^{(\theta-1)H}y^{-H}|y-x|^{2H-2}\ud x\ud y\\
&\leq \int_0^1\int_0^1 y^{\theta H - 1} z^{(\theta-1)H}(1-z)^{2H-2}\ud z \ud y\\
&\leq \frac{1}{\theta H}B((\theta-1)H+1,2H-1)
\end{split}
\end{equation*}
which converges to zero when divided with $T$ and let $T$ tends to infinity. The case $x>y$, can be treated in a similar way. Hence it is sufficient to consider the function 
$$
\frac{1}{2\theta}e^{-\theta|x-y|}
$$
instead of $\tilde{g}(x,y)$. We shall use L'Hopital's rule to compute the limit. Taking derivative with respect to $T$, we obtain
\begin{equation*}
\begin{split}
\frac{\ud I_T}{\ud T} &= \frac{\alpha_H^2H^{4H-4}}{\theta^2}\int_{[0,T]^{3}} e^{-\theta|T-u_1|}e^{-\theta|u_2-v_2|} e^{\left(\frac{1}{H}-1\right)(T+u_1+u_2+v_2)}\\
&\times \left|e^{\frac{u_2}{H}}-e^{\frac{u_1}{H}}\right|^{2H-2} \left|e^{\frac{T}{H}}-e^{\frac{v_1}{H}}\right|^{2H-2}\ud u_1\ud u_2\ud v_1\ud v_2.
\end{split}
\end{equation*}
By change of variables $x=T-u_1$, $y=T - u_2$ and $z=T-v_1$, this reduces to
\begin{equation*}
\begin{split}
\frac{\ud I_T}{\ud T} &= \frac{\alpha_H^2H^{4H-4}}{\theta^2}\int_{[0,T]^{3}} e^{-\theta x}e^{-\theta|y-z|} e^{\left(1-\frac{1}{H}\right)(x+y+z)}\\
&\times \left(1-e^{-\frac{y}{H}}\right)^{2H-2} \left|e^{-\frac{x}{H}}-e^{-\frac{z}{H}}\right|^{2H-2}\ud z\ud x\ud y.
\end{split}
\end{equation*}
Therefore, we have
\begin{equation*}
\begin{split}
\lim_{T \to \infty} \frac{\ud I_T}{\ud T} &= \frac{\alpha_H^2H^{4H-4}}{\theta^2}\int_{[0,\infty)^{3}} e^{-\theta x}e^{-\theta|y-z|} e^{\left(1-\frac{1}{H}\right)(x+y+z)}\\
&\times \left(1-e^{-\frac{y}{H}}\right)^{2H-2} \left|e^{-\frac{x}{H}}-e^{-\frac{z}{H}}\right|^{2H-2}\ud z\ud x\ud y.
\end{split}
\end{equation*}
We end the proof by showing that the later triple integral, denoted by $I$, is finite. Use the obvious bound  $e^{-\theta |z-y|} \le 1 $, we infer that
\begin{equation*}
\begin{split}
I &\leq \int_{[0,\infty)^{3}} e^{-\theta x} e^{\left(1-\frac{1}{H}\right)(x+y+z)}\\
&\times \left(1-e^{-\frac{y}{H}}\right)^{2H-2}\left|e^{-\frac{x}{H}}-e^{-\frac{z}{H}}\right|^{2H-2}\ud z\ud x\ud y\\
&= \left[\int_0^\infty e^{\left(1-\frac{1}{H}\right)y}\left(1-e^{-\frac{y}{H}}\right)^{2H-2}\ud y\right]\\
&\times \left[\int_0^\infty\int_0^\infty e^{-\theta x} e^{\left(1-\frac{1}{H}\right)(x+z)}\left|e^{-\frac{x}{H}}-e^{-\frac{z}{H}}\right|^{2H-2}\ud z\ud x\right]\\
&=I_1\times I_2.
\end{split}
\end{equation*}
For the term $I_1$, we obtain by change of variable $u=e^{-\frac{y}{H}}$ that
\begin{equation*}
I_1 = C\int_0^1 u^{-H}(1-u)^{2H-2}\ud u < \infty.
\end{equation*}
For the term $I_2$, we obtain by change of variables $u=e^{-\frac{x}{H}}$ and $v=e^{-\frac{z}{H}}$ that 
\begin{equation*}
\begin{split}
I_2 &= C\int_0^1\int_0^1 u^{(\theta-1)H}v^{-H}|u-v|^{2H-2}\ud u \ud v\\
&= \left[\int_0^1\int_0^u  + \int_0^1\int_u^1\right]u^{(\theta-1)H}v^{-H}|u-v|^{2H-2}\ud v\ud u\\
&= I_{2,1}+I_{2,2}.
\end{split}
\end{equation*}
For the term $I_{2,1}$, we obtain by change of variable $z=\frac{v}{u}$ that
\begin{equation*}
I_{2,1} = C \int_0^1 u^{\theta H-1} \int_0^1 z^{-H}(1-z)^{2H-2}\ud z \ud u = \frac{1}{\theta H} B(1-H,2H-1).
\end{equation*}
Similarly for the term $I_{2,2}$, we get by change of variable $z=\frac{u}{v}$ that
\begin{equation*}
I_{2,2} = \frac{1}{\theta H}B((\theta -1)H + 1, 2H-1).
\end{equation*}
\end{proof}

\end{document}